\documentclass[12pt,a4paper]{amsart}
\usepackage{amssymb}
\usepackage{hyperref}
\usepackage{amsthm}
\usepackage{amsmath}
\usepackage[T1]{fontenc}
\usepackage[utf8]{inputenc}
\newtheorem{theorem}{Theorem}[section]
\newtheorem{claim}[theorem]{Claim}

\newtheorem{lemma}[theorem]{Lemma}

\newtheorem{corollary}[theorem]{Corollary}
\theoremstyle{definition}
\newtheorem{definition}[theorem]{Definition}

\theoremstyle{remark}
\newtheorem{remark}[theorem]{Remark}
\newtheorem{question}[theorem]{Question}

\DeclareMathOperator{\CH}{CH}
\DeclareMathOperator{\GCH}{GCH}
\newcommand{\dom}[1]{{\rm dom}(#1)}

\newcount\skewfactor
\def\mathunderaccent#1#2 {\let\theaccent#1\skewfactor#2
\mathpalette\putaccentunder}
\def\putaccentunder#1#2{\oalign{$#1#2$\crcr\hidewidth
\vbox to.2ex{\hbox{$#1\skew\skewfactor\theaccent{}$}\vss}\hidewidth}}

\title[Adding highly generic subsets of $\omega_2$]{Adding highly generic subsets of $\omega_2$}
\author{Rouholah Hoseini Naveh}
\address{Department of Pure Mathematics\\
Faculty of Mathematics \& Computer\\
Shahid Bahonar University of Kerman\\
Kerman, Iran}
\email{r.hoseini.nave@gmail.com}

\author{Mohammad Golshani}
\address{School of Mathematics\\
Institute for Research in Fundamental Sciences (IPM)\\
P.O. Box:
19395-5746\\
Tehran-Iran.}
\email{golshani.m@gmail.com}
\urladdr{http://math.ipm.ac.ir/~golshani/}

\author{Esfandiar Eslami}
\address{Department of Pure Mathematics\\
Faculty of Mathematics \& Computer\\
Shahid Bahonar University of Kerman\\
Kerman, Iran}
\email{eeslami@mail.uk.ac.ir}

\thanks{
 The first author's research is partially supported by FWF project V844, hosted by Sandra M\"{u}ller. 
The second author's research has been supported by a grant from IPM (No. 1401030417). The authors thank Rahman Mohammadpour for some useful comments and suggestions}
\keywords {Side condition forcing, Matrices of elementary substructures}
\begin{document}
\makeatletter
\makeatother
\maketitle
\begin{abstract}
Starting from the $\GCH,$ we build a cardinal and $\GCH$ preserving generic extension of the universe, in which there exists a set $A \subseteq \omega_2$ of size $\aleph_2$ so that every
countably infinite subset of $A$ or $\omega_2 \setminus A$ is Cohen generic over the ground model.
\end{abstract}
\section{Introduction}
It is clear that if $\kappa \geq \aleph_0$ is an infinite cardinal, then the Cohen forcing
 $$\mathbb{P}_{\kappa}=\{p: \kappa \longrightarrow 2 : |p|< \aleph_0 \}$$
forces the existence of a set $A \subseteq \kappa$ of size $\kappa$ such that $X \cap A$ and $X \setminus A$ are non-empty for all countably infinite ground model sets
$X \subseteq \kappa$. It also forces $2^{\aleph_0}\geq \kappa$, hence for $\kappa \geq \aleph_2$, the $\GCH$ fails in the extension.
In personal communication with the second author, Moti Gitik asked the following natural question:
 \begin{question}
 \label{q1}
 Suppose that the $GCH$ holds and $\kappa \geq \aleph_2$ is a cardinal. Is there a cardinal and $GCH$ preserving extension of the universe in which there exists a set $A \subseteq \kappa$ of size $\kappa$, such that for all countably infinite sets $X \in \mathcal{P}(\kappa) \cap V$, $X \cap A$ and $X \setminus A$ are nonempty?
 \end{question}
 In this paper we use Todorcevic's method of forcing with matrices of countable elementary substructures to answer the question for the case $\kappa = \aleph_2$.
 
 \begin{theorem}
 \label{t1}
($\GCH$) There exists a cardinal and $\GCH$ preserving generic extension of the universe by a strongly proper forcing notion,
such that in the generic extension, there exists a set $A \subseteq \omega_2$ of size $\aleph_2$ such that for all countably infinite ground model sets $X \subseteq \omega_2$, $X \cap A$ and $X \setminus A$ are non-empty.
\end{theorem}
\begin{remark}
It follows from the proof of Theorem \ref{t1} that the set $A$ above also satisfies the conclusion of the abstract, namely every countably infinite subset of $A$ or $\omega_2 \setminus A$ is Cohen generic over the ground model.
 \end{remark}
 The paper is organized as follows. In section \ref{pre}, we recall some basic definitions and results about strongly proper forcing notions,
 and review Todorcevic's matrix $\in$-collapse forcing, then in section \ref{proof}, we present a proof of Theorem \ref{t1}.

\section{Some preliminaries}
\label{pre}
The notion of proper forcing was introduced by Shelah, see \cite{shelah80}, who showed that proper forcing notions preserve
$\aleph_1$ and that their countable support iteration is again proper. In this paper we work with a stronger concept, called strongly proper,
which was introduced by Mitchell \cite{mitchell}.
\begin{definition}
\label{d1} Let $\mathbb{P}$ be a forcing notion and $X$ be a set.
\begin{enumerate}
\item We say that $p$ is \textit{strongly $(X, \mathbb{P})$-generic} if for any set $D$ which is dense and open in the poset $\mathbb{P} \cap X$, the set $D$ is predense in $\mathbb{P}$ below $p$.
\item The poset $\mathbb{P}$ is \textit{strongly proper} if for every large enough regular cardinal $\theta$, there are club many countable elementary submodels $M$ of $H(\theta)$ such that whenever $p \in M \cap \mathbb{P}$, there exists a strongly $(M, \mathbb{P})$-generic condition below $p$.
\end{enumerate}
\end{definition}
The following easy lemma gives a characterization of strongly $(M, \mathbb{P})$-generic conditions.
\begin{lemma}
\label{l1} Let $\mathbb{P}, \theta$ and $M$ be as above and let $p \in \mathbb{P}.$
Then $p$ is strongly $(M, \mathbb{P})$-generic iff
\begin{itemize}
\item[$(*)$:] For every $q \leq p,$ there is $q\mid_M \in \mathbb{P} \cap M$ such that for every $r \in \mathbb{P} \cap M,$
if $r \leq q\mid_M$, then $r$ and $q$ are compatible in $\mathbb{P}$.
\end{itemize}
\end{lemma}
The method of forcing with side conditions was introduced by Todorcevic \cite{todorcevic1}, who used it to get several consequences of the proper forcing axiom. He also introduced a variant of his method, where the side conditions form a matrix, and not a chain, see \cite{todorcevic3} and \cite{todorcevic2} for a detailed exposition of the method. We will use matrix side conditions as working parts of our forcing notion to prove Theorem \ref{t1}.

Fix a well-ordering $\lhd$ of $H(\omega_2)$.
Throughout this paper, by $M \prec H(\omega_2)$ we mean $\langle M, \in, \lhd \cap M^2 \rangle$ is an elementary substructure of $\langle H(\omega_2), \in, \lhd \rangle$. Set
$$\mathcal{S}=\{M \in [H_{\omega_2}]^{\aleph_0}: M \prec H_{\omega_2} \}.$$
Note that $\mathcal{S}$ is a club subset of $[H(\omega_2)]^{\aleph_0}$. For every $M, N \in \mathcal{S}$, we write $M\cong N$ if and only if $\langle M, \in\rangle$ is isomorphic to $\langle N, \in \rangle$, and denote the unique isomorphism between them by $\varphi_{M,N}:M \xrightarrow{\simeq} N$.
 For each $M \in \mathcal{S}$ we denote $M \cap \omega_1$ by $\delta_M$, $M\cap \omega_2$ by $\beta_M$, and if $p \subseteq \mathcal{S}$, we let $\dom{p}=\{\delta_M: M\in p \}$. Also $p(\delta)$ denotes the set of all $M \in p$ with $\delta_M=\delta$. We are now ready to define
 the matrix $\in$-collapse forcing.

\begin{definition}\label{deftod}
The forcing notion $\mathbb{Q}$ consists of all finite $p\subset \mathcal{S}$ satisfying the following conditions:
\begin{itemize}
\item[$(1)$] If $M,N \in p$ and $\delta_M=\delta_N$, then $ M \cong N$;
\item[$(2)$] If $M \in p$ and $\delta \in \dom{p}$ is such that $\delta_M < \delta$, then there exists $N \in p(\delta)$ such that $M \in N$.
\end{itemize}
For $p, q \in \mathbb{Q}$, we say $p \leq q$ if $q \subseteq p$.
\end{definition}
We may remark that  the forcing notion $\mathbb{Q}$ defined above  is equivalent to the forcing notion $\mathcal{P}$ from \cite[Definition 2.1]{todorcevic3}.  To see this, in the definition $2.1$ of \cite{todorcevic3}, set  $\theta = \omega_2$, and   note that using the notations $\dom{p}$ and $p(\delta)$ given above, we can easily check that every condition $p \in \mathbb{Q}$, is indeed a finite function $p \colon \omega_1 \longrightarrow H(\omega_2)$ which satisfies all the required items for poset $\mathcal{P}$ from  \cite[Definition 2.1]{todorcevic3}.
We have the following lemma.
\begin{lemma}
\label{l2} ($\GCH$) The forcing notion $\mathbb{Q}$ is strongly proper,
satisfies the $\aleph_2$-c.c., and preserves the $\GCH.$
\end{lemma}
\begin{proof}
See \cite{todorcevic3}
\end{proof}

\section{Proof of Theorem \ref{t1}}
\label{proof}
In this section we prove our main theorem, by introducing a strongly proper forcing notion which preserves the $\GCH$ and adds a set $A \subseteq \omega_2$ as requested.
Let us start by defining our forcing notion.
\begin{definition}
\label{d2}
A pair $p=\langle \mathcal{M}_p, f_p \rangle$ is a condition of $\mathbb{P}$ whenever:
\begin{itemize}
\item[$(i)$] $\mathcal{M}_p \in \mathbb{Q}$;
\item[$(ii)$] $f_p: \omega_2 \longrightarrow 2$ is a finite partial function; and
\item[$(iii)$] if $M,N \in \mathcal{M}_p$ with $\delta_M=\delta_N$, then
 \begin{itemize}
\item $\alpha \in (\dom{f_p} \cap M) \Rightarrow \varphi_{M,N}(\alpha) \in \dom{f_p}$,
\item for each $\alpha$ as above, $f_p(\varphi_{M,N}(\alpha))=f_p(\alpha)$.
\end{itemize}
\end{itemize}
For $p, q \in \mathbb{P}$, we say $p \leq q$ if and only if $\mathcal{M}_q \subseteq \mathcal{M}_p$ and $f_q \subseteq f_p$.
\end{definition}
 The following lemma plays a key role
in the verification of strong properness of $\mathbb{P}$.
\begin{lemma}\label{mhatlemma}
Let $\theta > \omega_2$ be a large enough regular cardinal and let $M \prec H(\theta)$ be countable. Let $p=\langle \mathcal{M}_p, f_p \rangle \in \mathbb{P}$ such that $M \cap H(\omega_2)=M_0 \in \mathcal{M}_p$. Then there are $\hat{\mathcal{M}}_p$ and $\hat{f}_p$ which satisfy the following conditions:
\begin{itemize}
\item[$(1)$] $\hat{\mathcal{M}}_p \in \mathbb{Q} \cap M$;
\item[$(2)$] $\dom{\hat{\mathcal{M}}_p}= \dom{\mathcal{M}_p} \cap M$;
\item[$(3)$] $\mathcal{M}_p \cap M \subseteq \hat{\mathcal{M}}_p$;
\item[$(4)$] if $\alpha \in \dom{\hat{\mathcal{M}}_p}, N_1 \in \mathcal{M}_p(\alpha)$ and $N_2 \in \hat{\mathcal{M}}_p(\alpha)$, then $N_1 \cong N_2$;
\item[$(5)$] $\hat{\mathcal{M}}_p \cup \mathcal{M}_p$ is $M_0$-full, i.e. for every $N\in(\hat{\mathcal{M}}_p \cup \mathcal{M}_p) \cap M_0$, for every $\delta \in \dom{\hat{\mathcal{M}}_p \cup \mathcal{M}_p}$, with $\delta_N < \delta < \delta_{M_0}$, there exists $K\in (\hat{\mathcal{M}}_p \cup \mathcal{M}_p)(\delta)$ such that $N \in K \in M_0$;
 \item[$(6)$] $\hat{f}_p \supseteq f_p \restriction M;$
\item[$(7)$] $\hat{p}= \langle \hat{\mathcal{M}}_p, \hat{f}_p \rangle \in \mathbb{P} \cap M$; and
\item[$(8)$] $\hat{p}$ and $p$ are compatible.
\end{itemize}
\end{lemma}
\begin{proof}
 Note that if $\delta \in \dom{\mathcal{M}_p} \cap M$
 and $N \in \mathcal{M}_p(\delta)$, then $N\cong N'$ for some $N'\in M_0$. Thus by elementarity, we can obtain $\hat{\mathcal{M}}_p$, satisfying items $(1)$-$(4)$. It is then clear that $\hat{\mathcal{M}}_p \cup \mathcal{M}_p$ is $M_0$-full,
 hence clause $(5)$
 is satisfied as well. Furthermore note that $\hat{\mathcal{M}}_p$ and $\mathcal{M}_p$ are compatible.

 Now set $\hat{f}_p=f_p \restriction M.$ Then items $(6)$ and $(8)$ are satisfied trivially, so we are left to show that
$\hat{p}= \langle \hat{\mathcal{M}}_p, \hat{f}_p \rangle$ is indeed a condition.
We just need to show that if $N_1, N_2 \in \hat{\mathcal{M}}_p, \delta_{N_1}=\delta_{N_2}$ and $\alpha \in \dom{f_p{\restriction M}} \cap N_1$, then $\varphi_{N_1,N_2}(\alpha) \in \dom{f_p{\restriction M}}$ and $f_p(\varphi_{N_1,N_2}(\alpha))=f_p(\alpha)$.
Fix $N_1, N_2$ and $\alpha$ as above. Note that $\alpha \in M.$

Since $N_1, N_2 \in M, \varphi_{N_1,N_2} \in M$ and hence $\varphi_{N_1,N_2}(\alpha) \in M.$
 Now $\alpha \in N_1 \in \hat{\mathcal{M}}_p$ implies that there exist $x$, $N_1'$ and $M_1'$ such that:
 \begin{enumerate}
 \item $x \in N_1' \in M_1' \in \mathcal{M}_p(\delta_{M_0})$,
 \item $N_1= \varphi_{M_1', M_0}(N_1')$, and
 \item $\alpha= \varphi_{M_1', M_0}(x)$.
 \end{enumerate}
 Then by Definition \ref{d2}(iii),
 \begin{center}
 $x \in \dom{f_p}$ and $f_p(x)=f_p(\alpha)$.
 \end{center}
 Also $\varphi_{N_1,N_2}(\alpha) \in N_2 \in \hat{\mathcal{M}}_p$ implies that there exist
 $y, N_2'$ and $M_2'$ such that
\begin{enumerate}
\item[(4)] $y \in N_2' \in M_2' \in \mathcal{M}_p(\delta_{M_0})$,
\item[(5)] $N_2= \varphi_{M_2', M_0}(N_2')$, and
\item[(6)] $\varphi_{N_1,N_2}(\alpha)= \varphi_{M_2', M_0}(y)$.
\end{enumerate}
Then $\varphi_{M_1', M_2'}(x)=y$, so
by Definition \ref{d2}(iii),
\begin{center}
$y \in \dom{f_p}$ and $f_p(y)=f_p(x)$.
\end{center}
By clause (6), $\varphi_{N_1,N_2}(\alpha) \in \dom{f_p}$ and
\[
f_p(\varphi_{N_1,N_2}(\alpha) )=f_p(\varphi_{M_2', M_0}(y))=f_p(y)=f_p(x)=f_p(\alpha).
\]
 The lemma follows.
\end{proof}
The next lemma is standard.
\begin{lemma}\cite[Lemma $2.7$]{todorcevic3} \label{xi}
If $N_0$ and $N_1$ are two isomorphic elementary substructures of $H(\omega_2)$ and $\beta \in N_0 \cap N_1 \cap \omega_2$, then for all $\xi < \beta, \xi \in N_0$ if and only if $\xi \in N_1$.
\end{lemma}
\begin{proof}
For each $\xi \in \omega_2$, there is a $1-1$-function from $\xi$ into $\omega_1$. Without loss of generality we can assume that both $N_0$ and $N_1$ contain the same family of mappings $\langle e_{\gamma}: \gamma \in \omega_2 \rangle$ where $e_{\gamma}: \gamma \longrightarrow \omega_1$ is a $1-1$-function. Let $\xi < \beta$ and $\xi \in N_0$, so $e_{\beta}(\xi) \in N_0 \cap \omega_1=N_1 \cap \omega_1$. Hence $\xi=e_{\beta}^{-1}(e_{\beta}(\xi)) \in N_1$.
\end{proof}
\begin{definition}
Let $X$ be a set. A finite subset $w=\{x_0, x_1, \dots, x_{n-1} \} \subseteq X$ is called an $X$-\textit{path}, if $x_i \in x_{i+1}$ for all $0 \leq i < n-1$. We may use \textit{path} instead of $X$-path, when the set $X$ is evident from the context.
\end{definition}
The next lemma guarantees the existence of natural strongly $(M, \mathbb{P})$-generic conditions.
\begin{lemma}
\label{l22}
Let $\theta > \omega_2$ be a large enough regular cardinal and let $M \prec H_{\theta}$ be countable. If $p=\langle \mathcal{M}_p, f_p \rangle \in \mathbb{P}$ with $M \cap H(\omega_2) = M_0 \in \mathcal{M}_p$ and $f_p \in M$, then $p$ is a strongly $(M, \mathbb{P})$-generic condition.
\end{lemma}
\begin{proof}
Let $D$ be an open dense subset of $\mathbb{P} \cap M$ and $q=\langle \mathcal{M}_q, f_q \rangle \leq p$. We have to show that $q$ is compatible with some element of $D$. Let $\hat{\mathcal{M}}_q$ be as in Lemma \ref{mhatlemma}. Let $X$ be the set of all $(\hat{\mathcal{M}}_q \cup \mathcal{M}_q)$-paths $w=\{N_0^w, \dots, N_l^w\}$ such that $N_l^w \in \mathcal{M}_q(\delta_{M_0})$, which gives $N_l^w$ and $M_0$ are isomorphic. Set $$\mathcal{M}_{q \restriction M} =\{ \varphi_{N_l^w, M_0}(N_i^w): w=\langle N_0^w, \dots, N_l^w \rangle \in X \land i <l \}.$$

Note that by the construction, $\hat{\mathcal{M}}_q \subseteq \mathcal{M}_{q \restriction M}$, $\dom{\mathcal{M}_{q \restriction M}}=\dom{\mathcal{M}_q} \cap M$ and it is easy to see that $\mathcal{M}_{q \restriction M} \in \mathbb{Q} \cap M$.
By an argument similar to the proof of Lemma \ref{mhatlemma}, $q{\restriction M}=\langle \mathcal{M}_{q \restriction M}, f_q{\restriction \beta_M} \rangle \in \mathbb{P} \cap M$.
Since $D \subseteq \mathbb{P} \cap M$ is an open dense set, we can find $r \in D \cap M$ such that $r \leq q{\restriction M}$. We now define $\bar{q}=\langle \bar{\mathcal{M}}, \bar{f} \rangle$ where:
\begin{itemize}
 \item $\bar{\mathcal{M}}=\mathcal{M}_r \cup \mathcal{M}_q \cup \{\varphi_{M_0, N}(K): N \in \mathcal{M}_q(\delta_{M_0}) \land K \in \mathcal{M}_r \}$, and
 \item $\bar{f}=f_r \cup f_q \cup \{\langle \varphi_{N', N^{\prime\prime}}(\alpha), f_r(\alpha) \rangle :\alpha \in \dom{f_r}, N' , N^{\prime\prime} \in \bar{\mathcal{M}} \land \delta_{N^\prime}=\delta_{N^{\prime\prime}}\}.$
\end{itemize}
\begin{claim}
 $\bar{f}:\omega_2 \longrightarrow 2$ is a finite partial function.
 \end{claim}
\begin{proof} Let $x_1, x_2 \in \dom{\bar{f}}$ with $x_1=x_2$. Note that $f_q{\restriction {\beta_M}} \subseteq f_r$ and $\dom{f_r} \subseteq \beta_M$, so it is enough to consider the following two cases.

{\bf Case 1}: $x_1=\varphi_{N_0',N_0^{\prime\prime}}(\alpha_1)$ for some $\alpha_1 \in \dom{f_r}$ and $x_2=\varphi_{N_1', N_1^{\prime\prime}}(\alpha_2)$ for some $\alpha_2 \in \dom{f_r}$ where $\delta_{N_0'}=\delta_{N_0^{\prime\prime}}$ and $\delta_{N_1'}=\delta_{N_1^{\prime\prime}}$. Without loss of generality we can assume that $\delta_{N_0^{\prime\prime}}=\delta_{N_1^{\prime\prime}}$, since otherwise, suppose that $\delta_{N_0^{\prime\prime}} < \delta_{N_1^{\prime\prime}}$. So for some $N_2', N_2^{\prime\prime} \in \bar{\mathcal{M}}(\delta_{N_1^{\prime\prime}})$, $N_0' \in N_2'$, $N_0^{\prime\prime} \in N_2^{\prime\prime}$ and $\varphi_{N_2',N_2^{\prime\prime}}\restriction {N_0'}=\varphi_{N_0',N_0^{\prime\prime}}$.
 Thus $x_1=\varphi_{N_2',N_2^{\prime\prime}}(\alpha_1)$, where $\delta_{N_2^{\prime\prime}}=\delta_{N_1^{\prime\prime}}$,
 and we may replace $(N_0',N_0^{\prime\prime})$ by $(N_2',N_2^{\prime\prime})$.
 Then $\varphi_{N_0^\prime, N_1^\prime}(\alpha_1)=\alpha_2$, and hence
$$\bar{f}(x_1)=f_r(\alpha_1)=f_r(\varphi_{N_0^\prime, N_1^\prime}(\alpha_1))=f_r(\alpha_2)=\bar{f}(x_2).$$

{\bf Case 2}: $x_1 \in \dom{f_r}$ and $x_2=\varphi_{N_0',N_0^{\prime\prime}}(\alpha)$ for some $\alpha \in \dom{f_r}$. Again we can assume that for some $N' \in \bar{\mathcal{M}}(\delta_{N_0^{\prime\prime}})$ we have $x_1 \in N'$ and $x_2=\varphi_{N', N_0^{\prime\prime}}(x_1)$. Thus
$$\bar{f}(x_1)=f_r(x_1)= f_r(\varphi_{N', N_0^{\prime\prime}}(x_1)) =f_r(x_2)=\bar{f}(x_2).$$
\end{proof}
\begin{claim}
 $\bar{\mathcal{M}} \in \mathbb{Q}$.
 \end{claim}
\begin{proof} It suffices to show that $\bar{\mathcal{M}}$
 satisfies clause (2) of Definition \ref{deftod}. Suppose that $\alpha < \beta < \omega_1$ are in $\dom{\bar{\mathcal{M}}}$ and $N \in \bar{\mathcal{M}}(\alpha)$. If $\alpha \geq \delta_{M_0}$, then $\bar{\mathcal{M}}(\alpha)=\mathcal{M}_q(\alpha)$ and $\bar{\mathcal{M}}(\beta)= \mathcal{M}_q(\beta)$, hence there is $N' \in \mathcal{M}_q(\beta) = \bar{\mathcal{M}}(\beta)$ such that $N \in N'$.

Now suppose that $\alpha < \delta_{M_0}$. There are three subcases, depending on the relation between $\beta$ and $\delta_{M_0}$.

First suppose that $\beta < \delta_{M_0}$. If $N \in \mathcal{M}_r(\alpha)$, then we can find some $N' \in \mathcal{M}_r(\beta) \subseteq \bar{\mathcal{M}}(\beta)$ such that $N \in N'$ and we are done. Otherwise, $N=\varphi_{M_0, N'}(K),$ where $N' \in \mathcal{M}_q(\delta_{M_0})$ and $K \in \mathcal{M}_r$.
$N \in \bar{\mathcal{M}}(\alpha)$. Let $K^\prime \in \mathcal{M}_r(\beta)$ be such that $K \in K^\prime.$ Then $\varphi_{M_0, N'}(K^\prime) \in \bar{\mathcal{M}}(\beta)$
and $N \in \varphi_{M_0, N'}(K^\prime).$

Now let $\beta = \delta_{M_0}$. If $N \in \mathcal{M}_r(\alpha)$, then $N \in M_0 \in \bar{\mathcal{M}}(\beta)$, otherwise, $N=\varphi_{M_0, N'}(K),$ where $N' \in \mathcal{M}_q(\delta_{M_0})$ and $K \in \mathcal{M}_r$. But then $N \in N'$ and we are done again.

Finally suppose that $\beta > \delta_{M_0}$. Then we can find some $N' \in \bar{\mathcal{M}}(\delta_{M_0})$ and $N'' \in \bar{\mathcal{M}}(\beta)$
such that $N \in N' \in N''$. Thus $N \in N''$ and we are done
\end{proof}
\begin{claim}
 $\bar{q}= \langle \bar{\mathcal{M}}, \bar{f} \rangle \in \mathbb{P}$.
 \end{claim}
 \begin{proof}
By the previous claims, $\bar{f}$ is a finite partial function and $\bar{\mathcal{M}} \in \mathbb{Q}$. Let $\alpha \in N_1 \cap \dom{\bar{f}}$ and $N_1 \cong N_2$,
where $N_1, N_2 \in \bar{\mathcal{M}}$. We have to show that $ \varphi_{N_1, N_2}(\alpha) \in \dom{\bar{f}}$
and $\bar{f}(\varphi_{N_1, N_2}(\alpha))=\bar{f}(\alpha)$.

 If $\alpha \in \dom{f_r}$, then $\langle \varphi_{N_1, N_2}(\alpha), f_r(\alpha) \rangle \in \bar{f}$, and we are done. If $\alpha \in \dom{f_q} \setminus \dom{f_r}$, then $\alpha \notin \beta_M$ and we must have $N_1, N_2 \in \mathcal{M}_q \setminus \mathcal{M}_r$. So $\varphi_{N_1, N_2}(\alpha) \in \dom{f_q}$ and $f_q(\varphi_{N_1, N_2}(\alpha))=f_q(\alpha)$, which implies $\varphi_{N_1, N_2}(\alpha) \in \dom{\bar{f}}$ and $\bar{f}(\varphi_{N_1, N_2}(\alpha))=\bar{f}(\alpha)$. Finally if $\alpha=\varphi_{N', N_1}(\beta)$ for some $\beta \in \dom{f_r}$ and $N^\prime \in \bar{\mathcal{M}}$ with $N^\prime \simeq N_1$, then
 $ \varphi_{N_1, N_2}(\alpha)=\varphi_{N_1, N_2}\varphi_{N', N_1}(\beta)= \varphi_{N^\prime, N_2}(\beta) \in \dom{\bar{f}}$ and
 \[
 \bar{f}(\varphi_{N_1, N_2}(\alpha))=\bar{f}(\varphi_{N^\prime, N_2}(\beta))=f_r(\beta)=f_r(\alpha)=\bar{f}(\alpha).
 \]
 The claim follows.
 \end{proof}
It is evident that $\bar{q}$ extends both $q$ and $r$, and hence $q$ and $r$ are compatible. The lemma follows.
\end{proof}
We have the following easy lemma.
\begin{lemma}
\label{l3}
Let $\theta > \omega_2$ be a large enough regular cardinal and $M\prec H(\theta)$ countable. If $p \in \mathbb{P} \cap M$, then $p'=\langle \mathcal{M}_p \cup \{M\cap H(\omega_2)\}, f_p \rangle$ is a condition.
\end{lemma}
Putting all things together, we get the following
\begin{corollary}
\label{l31}
$\mathbb{P}$ is strongly proper
\end{corollary}
\begin{proof}
Let $\theta$ be large enough regular, $M \prec H(\theta)$ with $\mathbb{P} \in M$ and let $p \in \mathbb{P} \cap M$. Set $p'=\langle \mathcal{M}_p \cup \{M\cap H(\omega_2)\}, f_p \rangle$. By Lemma \ref{l3}, $p^\prime$ is a condition and by Lemma
\ref{l22}, $p^\prime$ is a strongly $(M, \mathbb{P})$-generic condition.
\end{proof}
In particular, it follows that forcing with $\mathbb{P}$ preserves $\aleph_1$.
\begin{lemma}
\label{l4}
$\mathbb{P}$ satisfies the $\aleph_2$-c.c.
\end{lemma}
\begin{proof}
Let $\{p_{\alpha}=\langle \mathcal{M}_{\alpha}, f_{\alpha} \rangle: \alpha < \omega_2\}$ be a collection of conditions. For each $\alpha < \omega_2, \dom{f_{\alpha}}$ is a finite subset of $\omega_2$, so by the $\Delta$-system lemma, we may assume that $\{\dom{f_{\alpha}}:\alpha < \omega_2\}$ forms a $\Delta$-system with root $d \subset \omega_2$, so that for every $\alpha \neq \beta, \dom{f_{\alpha}} \cap \dom{f_{\beta}}=d$. Since there are only finitely many functions $f: d \longrightarrow 2$,
by shrinking the sequence, we may also assume that $f_{\alpha}{\restriction d}=g$ for some fixed $g:d \longrightarrow 2$ and all $\alpha < \omega_2$.

For each $\alpha < \omega_2$ set
$$\bar{\mathcal{M}}_{\alpha}=\{\bar{M}: \exists M \in \mathcal{M}_{\alpha}, \bar{M} \ \text{is the transitive collapse of}\ M\} \in H(\omega_1).$$
Clearly for every $\alpha < \beta < \omega_2$, if $\bar{\mathcal{M}}_{\alpha}=\bar{\mathcal{M}}_{\beta}$, then $\mathcal{M}_{\alpha} \cup \mathcal{M}_{\beta} \in \mathbb{Q}$.

By $\CH,$ $|H(\omega_1)|=\aleph_1$, so by shrinking the sequence of conditions further, we may assume that $\bar{\mathcal{M}}_{\alpha}=\bar{\mathcal{M}}_{\beta}$
for all $\alpha < \beta < \omega_2$.

We now show that for $\alpha < \beta < \omega_2$, the conditions $p_\alpha$ and $p_\beta$ are compatible. Thus fix $\alpha < \beta < \omega_2$.
Let $q=\langle \mathcal{M}_q, f_q \rangle$ where
\begin{itemize}
\item $\mathcal{M}_q=\mathcal{M}_{p_{\alpha}} \cup \mathcal{M}_{p_{\beta}}$, and
 \item $f_q=f_{p_{\alpha}} \cup f_{p_{\beta}} \cup \{\langle \varphi_{N,N'}(\gamma), (f_{p_{\alpha}} \cup f_{p_{\beta}})(\gamma) \rangle : N, N' \in \mathcal{M}_q \land N \cong N' \land \gamma \in \dom{f_{p_{\alpha}} \cup f_{p_{\beta}}} \}$.
 \end{itemize}
 It is easily seen that $q$ is a condition which extends both $p_{\alpha}$ and $p_{\beta}$.
\end{proof}
Using corollary \ref{l31} and Lemma \ref{l4}, we get the following.
\begin{corollary}
\label{c1}
The forcing notion $\mathbb{P}$ preserves all cardinals.
\end{corollary}
We now show that forcing with $\mathbb{P}$ preserves the $\GCH$. We only need to consider the case of $\CH.$
\begin{lemma}
 Suppose $G$ is a $V$-generic filter over $ \mathbb{P}$. Let $M$ and $M'$ be countable isomorphic elementary substructures of $H(\theta)$, for a large enough regular cardinal $\theta$, with $\mathbb{P} \in M \cap M'$, and let $p \in \mathbb{P} \cap M$. Set $M_0=M \cap H(\omega_2)$ and $M_0^\prime=M' \cap H(\omega_2)$. Then $$p_{MM'}=\langle \mathcal{M}_p \cup \{M_0, M_0'\}, f_p \cup \{\langle \varphi_{M',M}(\alpha), f_p(\alpha) \rangle : \alpha \in \dom{f_p} \cap M'\} \rangle$$ is a condition, and it forces $\check{\varphi}_{M, M'}[\dot{G} \cap \check{M}]=\dot{G} \cap \check{M'}$.
\end{lemma}
\begin{proof}
We can easily check that $p_{M,M'}$ is a condition. For the sake of contradiction suppose that there is a condition $q \leq p_{M, M'}$ and there is $p'$ such that $q \Vdash$``$ \check{p'} \in \dot{G} \cap \check{M} \ \text{but} \ \check{\varphi}_{M, M'}(\check{p'}) \notin \dot{G} \cap \check{M'}$''. Since $q \Vdash \check{p'} \in \dot{G}$, $q$ and $p'$ are compatible, so let $q'$ be a common extension of $q, p'$. If for all $r \leq q'$, there be some $t \leq r$ such that $t \leq \varphi_{M,M'} (p')$, it then follows that the set
$$\{t \in \mathbb{P}: t\Vdash \check{\varphi}_{M, M'}(\check{p'}) \in \dot{G} \}$$
 is dense below $q'$, which is impossible, because it would imply that $q' \Vdash$``$ \check{\varphi}_{M, M'}(\check{p'}) \in \dot{G} \cap \check{M'}$'', which contradicts our assumption. Hence we can pick some $r \leq q'$ such that for all $t \leq r, \neg\left(t \leq \varphi_{M,M'}(p')\right)$ i.e. $r$ is incompatible with $\varphi_{M, M'}(p')$. Now consider the condition $r{\restriction} M \in M$. we have the following easy claim.
\begin{claim}
$\varphi_{M, M'}(r{\restriction} M)=r{\restriction} {M'}$.
\end{claim}
Now since $r \leq p'$ and $p' \in M$, we have $r{\restriction} M \leq p'$. By applying $\varphi_{M, M^\prime}$, we get $r{\restriction} {M'}=\varphi_{M,M'}(r{\restriction} M) \leq \varphi_{M, M'}(p')$ and hence $r \leq \varphi_{M, M'}(p')$. But $r \perp \varphi_{M, M'}(p')$, which is a contradiction.
\end{proof}
The proof of the next lemma is standard, but we present it for completeness.
\begin{lemma}
Forcing with $\mathbb{P}$ preserves the $\CH$.
\end{lemma}
\begin{proof}
By contradiction suppose that $\langle r_{\alpha}: \alpha < \omega_2 \rangle$ is a sequence of pairwise distinct reals in $V[G]$, where $G \subseteq \mathbb{P}$
is $V$-generic. Let $p$ be a condition that force this statement. For each $\alpha < \omega_2$ let $p_{\alpha} \leq p$ force ``$\dot{r}_{\alpha} \subseteq \check{\omega}$ is a real''. Fix $\theta$ large enough and regular. For each $\alpha < \omega_2$, let $M_{\alpha}$ be a countable elementary substructure of $H(\theta)$ with $p_{\alpha}, p, \mathbb{P}, \dot{r}_{\alpha} \in M_{\alpha}$.

 By counting arguments, there are $ \alpha < \beta < \omega_2$ such that
 \[
 \langle M_\alpha, \in, \mathbb{P}, p_\alpha, \dot{r}_\alpha \rangle \simeq \langle M_\beta, \in, \mathbb{P}, p_\beta, \dot{r}_\beta \rangle.
 \]
 In particular,
 $\varphi_{M_{\alpha}, M_{\beta}}(\dot{r}_{\alpha})=\dot{r}_{\beta}$ and $\varphi_{M_{\alpha}, M_{\beta}}(p_{\alpha})=p_{\beta}$.
 Set
 \[
 p_{M_{\alpha}, M_{\beta}} = (\mathcal{M}_{p_\alpha} \cup \mathcal{M}_{p_\beta} \cup \{ M_\alpha \cap H(\omega_2), M_\beta \cap H(\omega_2) \}, f_{p_\alpha} \cup f_{p_\beta}).
 \]
 Then $p_{M_{\alpha}, M_{\beta}}$ is a condition which extends both $p_\alpha$ and $p_\beta$.
Note that for all $n<\omega$, for all $p' \in M_{\alpha} \cap \mathbb{P}$ and for all $\xi \in \{0, 1 \}$, $p' \Vdash \dot{r}_{\alpha}(\check{n})= \check{\xi}$ if and only if $ \varphi_{M_{\alpha}, M_{\beta}}(p') \Vdash \dot{r}_{\beta} (\check{n})= \check{\xi}$
\begin{claim}
$p_{M_{\alpha}, M_{\beta}} \Vdash$``$ \dot{r}_{\alpha}=\dot{r}_{\beta}$''.
\end{claim}
\begin{proof}
By contradiction assume that there exist $q \leq p_{M_{\alpha}, M_{\beta}}$ and $n< \omega$ such that $q \Vdash$``$ \dot{r}_{\alpha}(\check{n})=0 \land \dot{r}_{\beta}(\check{n})=1$''. Then by elementarity, there is some $r \in \mathbb{P} \cap M_{\alpha}$ such that $r \leq q{\restriction} {M_{\alpha}}$ and $r \Vdash$``$ \dot{r}_{\alpha}(\check{n})=0$''. We have $\varphi_{M_{\alpha}, M_{\beta}}(q{\restriction} {M_{\alpha}})= q{\restriction} {M_{\beta}}$, so $q{\restriction} {M_{\beta}}$ and $\varphi_{M_{\alpha}, M_{\beta}}(r)$ are compatible. Hence we can conclude that $q|| \varphi_{M_{\alpha}, M_{\beta}}(r)$. But then $\varphi_{M_{\alpha}, M_{\beta}}(r) \Vdash$``$\dot{r}_{\beta}(\check{n})=0$'' which is a contradiction with the fact that $q \Vdash$``$ \dot{r}_{\beta}(\check{n})=1$''.
\end{proof}
We get a contradiction. Thus forcing with $\mathbb{P}$ must preserve the $\CH$
and the lemma follows.
\end{proof}
Now let $G$ be a $\mathbb{P}$-generic filter over $V$ and set
\begin{center}
$A=\{\alpha : \exists p \in G(\alpha \in \dom{f_p} \land f_p(\alpha)=1) \}$.
\end{center}
Then $A$ is a subset of $\omega_2$ of size $\aleph_2$.
\begin{lemma}
Suppose $X \in \mathcal{P}(\omega_2) \cap V$ is a countably infinite set. Then the sets $X \cap A$ and $X \setminus A$ are non-empty.
\end{lemma}
\begin{proof}
 Set
 $D_{X}= \{p \in \mathbb{P}: \exists \alpha, \beta \in X \cap \dom{f_p} \big(f_p(\alpha)=1 \land f_p(\beta)=0 \big) \}.$ It suffices to show that the set $D_X$ is dense,
 since if $p \in G \cap D_X$ and $\alpha, \beta \in X \cap \dom{f_p}$ are such that $f_p(\alpha)=1$ and $f_p(\beta)=0 $, then $\alpha \in X \cap A$ and $\beta \in X \setminus A$.

 To show that $D_X$ is dense, let $p \in \mathbb{P}$ be an arbitrary condition. As $X$ is infinite and $\dom{f_p}$ is finite, we can find $\alpha, \beta \in X \backslash \dom{f_p}$ such that for all $N \cong N'$ in $\mathcal M_p$, $\varphi_{N, N'}(\alpha)\neq \beta$. Set
\[
q=\langle \mathcal{M}_p, f_p \cup \{\langle \varphi_{N,N'}(\alpha), 1 \rangle, \langle \varphi_{N, N'}(\beta), 0 \rangle : N, N' \in \mathcal{M}_p \land N \cong N' \} \rangle.
 \]
 $q$ is easily seen to be a condition. Furthermore, it extends $p$ and belongs to $D_X$, as requested.
\end{proof}
This completes the proof of Theorem \ref{t1}.

\end{document}